\documentclass[12pt,twoside]{amsart}
\usepackage{amsmath, amsthm, amscd, amsfonts, amssymb, graphicx}
\usepackage[bookmarksnumbered, plainpages]{hyperref}
\usepackage [latin1]{inputenc}
\newtheorem{thm}{Theorem}[section]

\newtheorem{defn}[thm]{Definition}

\numberwithin{equation}{section}

\begin{document}

\title{\bf Affine Ricci solitons of three-dimensional Lorentzian Lie groups}
\author{Yong Wang}

\thanks{{\scriptsize
\hskip -0.4 true cm \textit{2010 Mathematics Subject Classification:}
53C40; 53C42.
 }}

\maketitle

\begin{abstract}
 In this paper, we classify affine Ricci solitons associated to canonical connections and Kobayashi-Nomizu connections and perturbed canonical connections and perturbed Kobayashi-Nomizu connections on three-dimensional Lorentzian Lie groups with
 some product structure.
\end{abstract}

\vskip 0.2 true cm

%------------------------------------------------------------------------------------%

\pagestyle{myheadings}
\markboth{\rightline {\scriptsize Wang}}
         {\leftline{\scriptsize Affine Ricci solitons}}

\bigskip
\bigskip

%------------------------------------------------------------------------------------%
%------------------------------------------------------------------------------------%

\section{ Introduction}
\indent The concept of the Ricci soliton is introduced by Hamilton in \cite{Ha}, which ia a naturel generalization of Einstein metrics.
Study of Ricci soliton over different geometric spaces is one of interesting topics in geometry and mathematical physics. In particular, it has become more important after G. Perelman applied Ricci solitons to solve the long standing Poincare conjecture.  In \cite{SO1,SO2,W1,W2,QW,HDZ}, Einstein manifolds associated to affine connections (especially semi-symmetric metric connections
and semi-symmetric non-metric connections) were studied (see the definition 3.2 in \cite{W2} and the definition 3.1 in \cite{HDZ}). It is natural
to study Ricci solitons associated to affine connections. Affine Ricci solitons had been introduced and studied, for example, see \cite{Cr,HD,HPC,PY,SCB}. In \cite{Ca0},
 Calvaruso studied three-dimensional generalized Ricci solitons, both in Riemannian and Lorentzian settings. He determined their homogeneous models, classifying left-invariant generalized Ricci solitons on three-dimensional Lie groups. Then it is natural to classify affine Ricci solitons
on three-dimensional Lie groups. In \cite{ES}, Etayo and Santamaria studied some affine connections on manifolds with the product structure
or the complex structure. In particular, the canonical connection and the Kobayashi-Nomizu connection for a product structure was studied. In \cite{W3}, we introduced a particular product structure on three-dimensional Lorentzian Lie groups and we computed canonical connections and Kobayashi-Nomizu connections and their curvature on three-dimensional Lorentzian Lie groups with
  this product structure. We defined algebraic Ricci solitons associated to canonical connections and Kobayashi-Nomizu connections. We classified algebraic Ricci solitons associated to canonical connections and Kobayashi-Nomizu connections on three-dimensional Lorentzian Lie groups with
 this product structure. In this paper, we classify affine Ricci solitons associated to canonical connections and Kobayashi-Nomizu connections and perturbed canonical connections and perturbed Kobayashi-Nomizu connections on three-dimensional Lorentzian Lie groups with
 this product structure.\\
\indent In Section 2, we classify affine Ricci solitons associated to canonical connections and Kobayashi-Nomizu connections on three-dimensional Lorentzian Lie groups with
 this product structure.
In Section 3, we classify affine Ricci solitons associated to perturbed canonical connections and perturbed Kobayashi-Nomizu connections on three-dimensional Lorentzian Lie groups with
 this product structure.

%------------------------------------------------------------------------------------%

\vskip 1 true cm

\section{ Affine Ricci solitons associated to canonical connections and Kobayashi-Nomizu connections on three-dimensional Lorentzian Lie groups}

Three-dimensional Lorentzian Lie groups had been classified in \cite{Ca1,CP}(see Theorem 2.1 and Theorem 2.2 in \cite{BO}). Throughout this paper,
we shall by $\{G_i\}_{i=1,\cdots,7}$, denote the connected, simply connected three-dimensional Lie group equipped with a left-invariant Lorentzian metric $g$ and
having Lie algebra $\{\mathfrak{g}\}_{i=1,\cdots,7}$. Let $\nabla$ be the Levi-Civita connection of $G_i$ and $R$ its curvature tensor, taken with the convention
\begin{equation}
R(X,Y)Z=\nabla_X\nabla_YZ-\nabla_Y\nabla_XZ-\nabla_{[X,Y]}Z.
\end{equation}
The Ricci tensor of $(G_i,g)$ is defined by
\begin{equation}\rho(X,Y)=-g(R(X,e_1)Y,e_1)-g(R(X,e_2)Y,e_2)+g(R(X,e_3)Y,e_3),
\end{equation}
where $\{e_1,e_2,e_3\}$ is a pseudo-orthonormal basis, with $e_3$ timelike.
We define a product structure $J$ on $G_i$ by
\begin{equation}Je_1=e_1,~Je_2=e_2,~Je_3=-e_3,
\end{equation}
then $J^2={\rm id}$ and $g(Je_j,Je_j)=g(e_j,e_j)$. By \cite{ES}, we define the canonical connection and the Kobayashi-Nomizu connection as follows:
\begin{equation}\nabla^0_XY=\nabla_XY-\frac{1}{2}(\nabla_XJ)JY,
\end{equation}
\begin{equation}\nabla^{{\rm 1}}_XY=\nabla^0_XY-\frac{1}{4}[(\nabla_YJ)JX-(\nabla_{JY}J)X].
\end{equation}
We define\begin{equation}
R^0(X,Y)Z=\nabla^0_X\nabla^0_YZ-\nabla^0_Y\nabla^0_XZ-\nabla^0_{[X,Y]}Z,
\end{equation}
\begin{equation}
R^{{1}}(X,Y)Z=\nabla^{{ 1}}_X\nabla^1_YZ-\nabla^1_Y\nabla^1_XZ-\nabla^1_{[X,Y]}Z.
\end{equation}
The Ricci tensors of $(G_i,g)$ associated to the canonical connection and the Kobayashi-Nomizu connection
are defined by
\begin{equation}\rho^0(X,Y)=-g(R^0(X,e_1)Y,e_1)-g(R^0(X,e_2)Y,e_2)+g(R^0(X,e_3)Y,e_3),
\end{equation}
\begin{equation}\rho^1(X,Y)=-g(R^1(X,e_1)Y,e_1)-g(R^1(X,e_2)Y,e_2)+g(R^1(X,e_3)Y,e_3).
\end{equation}
Let
\begin{equation}\widetilde{\rho}^0(X,Y)=\frac{{\rho}^0(X,Y)+{\rho}^0(Y,X)}{2},
\end{equation}
and
\begin{equation}\widetilde{\rho}^1(X,Y)=\frac{{\rho}^1(X,Y)+{\rho}^1(Y,X)}{2}.
\end{equation}

Since $(L_Vg)(Y,Z):=g(\nabla_YV,Z)+g(Y,\nabla_ZV)$, we let
\begin{equation}
(L^j_Vg)(Y,Z):=g(\nabla^j_YV,Z)+g(Y,\nabla^j_ZV),
\end{equation}
 for $j=0,1$ and vector fields $V,Y,Z$.
\begin{defn}
$(G_i,g,J)$ is called the affine Ricci soliton associated to the connection $\nabla^0$ if it satisfies
\begin{equation}
(L^0_Vg)(Y,Z)+2\widetilde{\rho}^0(Y,Z)+2\lambda g(Y,Z)=0
\end{equation}
where $\lambda$ is a real number and $V=\lambda_1e_1+\lambda_2e_2+\lambda_3e_3$ and $\lambda_1,\lambda_2,\lambda_3$ are real numbers.
$(G_i,g,J)$ is called the affine Ricci soliton associated to the connection $\nabla^1$ if it satisfies
\begin{equation}
(L^1_Vg)(Y,Z)+2\widetilde{\rho}^1(Y,Z)+2\lambda g(Y,Z)=0
\end{equation}
\end{defn}
By (2.1) and Lemma 3.1 in \cite{BO}, we have for $G_1$, there exists a pseudo-orthonormal basis $\{e_1,e_2,e_3\}$ with $e_3$ timelike such that the Lie
algebra of $G_1$ satisfies
\begin{equation}
[e_1,e_2]=\alpha e_1-\beta e_3,~~[e_1,e_3]=-\alpha e_1-\beta e_2,~~[e_2,e_3]=\beta e_1+\alpha e_2+\alpha e_3,~~\alpha\neq 0.
\end{equation}
By (2.25) in \cite{W3}, we have
\begin{align}
&\widetilde{\rho}^0(e_1,e_1)=-\left(\alpha^2+\frac{\beta^2}{2}\right),~~\widetilde{\rho}^0(e_1,e_2)=0,\\\notag
&\widetilde{\rho}^0(e_1,e_3)=\frac{\alpha\beta}{4},~~\widetilde{\rho}^0(e_2,e_2)=-\left(\alpha^2+\frac{\beta^2}{2}\right),\\\notag
&\widetilde{\rho}^0(e_2,e_3)=\frac{\alpha^2}{2},~~\widetilde{\rho}^0(e_3,e_3)=0.
\end{align}
By Lemma 2.4 in \cite{W3} and (2.12), we have
\begin{align}
&(L^0_Vg)(e_1,e_1)=2\lambda_2\alpha,~~(L^0_Vg)(e_1,e_2)=-\lambda_1\alpha\\\notag
&(L^0_Vg)(e_1,e_3)=-\frac{\beta}{2}\lambda_2,~~(L^0_Vg)(e_2,e_2)=0,\\\notag
&(L^0_Vg)(e_2,e_3)=\frac{\beta}{2}\lambda_1,~~(L^0_Vg)(e_3,e_3)=0.
\end{align}
If $(G_1,g,J,V)$ is an affine Ricci soliton associated to the connection $\nabla^0$, then by (2.13), we have
\begin{align}
\left\{\begin{array}{l}
2\lambda_2\alpha-2\alpha^2-\beta^2+2\lambda=0,\\
\lambda_1\alpha=0,\\
-\beta\lambda_2+\alpha\beta=0,\\
-2\alpha^2-\beta^2+2\lambda=0,\\
\frac{\beta}{2}\lambda_1+\alpha^2=0,\\
\lambda=0.\\
\end{array}\right.
\end{align}
Solve (2.18), we have
\vskip 0.5 true cm
\begin{thm}
$(G_1,g,J,V)$ is not an affine Ricci soliton associated to the connection $\nabla^0$.
\end{thm}
By (2.33) in \cite{W3}, we have
\begin{align}
&\widetilde{\rho}^1(e_1,e_1)=-\left(\alpha^2+{\beta^2}\right),~~\widetilde{\rho}^1(e_1,e_2)=\alpha\beta,\\\notag
&\widetilde{\rho}^1(e_1,e_3)=-\frac{\alpha\beta}{2},~~\widetilde{\rho}^1(e_2,e_2)=-\left(\alpha^2+\beta^2\right),\\\notag
&\widetilde{\rho}^1(e_2,e_3)=\frac{\alpha^2}{2},~~\widetilde{\rho}^1(e_3,e_3)=0.
\end{align}
By Lemma 2.8 in \cite{W3} and (2.12), we have
\begin{align}
&(L^1_Vg)(e_1,e_1)=2\lambda_2\alpha,~~(L^1_Vg)(e_1,e_2)=-\lambda_1\alpha,\\\notag
&(L^1_Vg)(e_1,e_3)=\lambda_1\alpha-\beta\lambda_2,~~(L^1_Vg)(e_2,e_2)=0,\\\notag
&(L^1_Vg)(e_2,e_3)=\beta\lambda_1-\alpha\lambda_2-\alpha\lambda_3,~~(L^1_Vg)(e_3,e_3)=0.
\end{align}
If $(G_1,g,J,V)$ is an affine Ricci soliton associated to the connection $\nabla^1$, then by (2.14), we have
\begin{align}
\left\{\begin{array}{l}
\lambda_2\alpha-\alpha^2-\beta^2+\lambda=0,\\
-\lambda_1\alpha+2\alpha\beta=0,\\
\lambda_1\alpha-\beta\lambda_2-\alpha\beta=0,\\
-\alpha^2-\beta^2+\lambda=0,\\
\beta\lambda_1-\alpha\lambda_2-\alpha\lambda_3+\alpha^2=0,\\
\lambda=0.\\
\end{array}\right.
\end{align}
Solve (2.21), we have
\vskip 0.5 true cm
\begin{thm}
$(G_1,g,J,V)$ is not an affine Ricci soliton associated to the connection $\nabla^1$.
\end{thm}
By (2.2) and Lemma 3.5 in \cite{BO}, we have for $G_2$, there exists a pseudo-orthonormal basis $\{e_1,e_2,e_3\}$ with $e_3$ timelike such that the Lie
algebra of $G_2$ satisfies
\begin{equation}
[e_1,e_2]=\gamma e_2-\beta e_3,~~[e_1,e_3]=-\beta e_2-\gamma e_3,~~[e_2,e_3]=\alpha e_1,~~\gamma\neq 0.
\end{equation}
By (2.44) in \cite{W3}, we have for $(G_2,g,J,\nabla^0)$
\begin{align}
&\widetilde{\rho}^0(e_1,e_1)=-\left(\gamma^2+\frac{\alpha\beta}{2}\right),~~\widetilde{\rho}^0(e_1,e_2)=0,\\\notag
&\widetilde{\rho}^0(e_1,e_3)=0,~~\widetilde{\rho}^0(e_2,e_2)=-\left(\gamma^2+\frac{\alpha\beta}{2}\right),\\\notag
&\widetilde{\rho}^0(e_2,e_3)=\frac{\beta\gamma}{2}-\frac{\alpha\gamma}{4},~~\widetilde{\rho}^0(e_3,e_3)=0.
\end{align}
By Lemma 2.14 in \cite{W3} and (2.12), we have for $(G_2,g,J,\nabla^0,V)$
\begin{align}
&(L^0_Vg)(e_1,e_1)=0,~~(L^0_Vg)(e_1,e_2)=\lambda_2\gamma\\\notag
&(L^0_Vg)(e_1,e_3)=-\frac{\alpha}{2}\lambda_2,~~(L^0_Vg)(e_2,e_2)=-2\gamma\lambda_1,\\\notag
&(L^0_Vg)(e_2,e_3)=\frac{\alpha}{2}\lambda_1,~~(L^0_Vg)(e_3,e_3)=0.
\end{align}
If $(G_2,g,J,V)$ is an affine Ricci soliton associated to the connection $\nabla^0$, then by (2.13), we have
\begin{align}
\left\{\begin{array}{l}
-\left(\gamma^2+\frac{\alpha\beta}{2}\right)+\lambda=0,\\
\lambda_2\gamma=0,\\
\alpha\lambda_2=0,\\
-\gamma\lambda_1-\left(\gamma^2+\frac{\alpha\beta}{2}\right)+\lambda=0,\\
\frac{\alpha}{2}\lambda_1+2(\frac{\beta\gamma}{2}-\frac{\alpha\gamma}{4})=0,\\
\lambda=0.\\
\end{array}\right.
\end{align}
Solve (2.25), we have
\vskip 0.5 true cm
\begin{thm}
$(G_2,g,J,V)$ is not an affine Ricci soliton associated to the connection $\nabla^0$.
\end{thm}
By (2.54) in \cite{W3}, we have for $(G_2,g,J,\nabla^1)$
\begin{align}
&\widetilde{\rho}^1(e_1,e_1)=-\left(\beta^2+\gamma^2\right),~~\widetilde{\rho}^1(e_1,e_2)=0,\\\notag
&\widetilde{\rho}^1(e_1,e_3)=0,~~\widetilde{\rho}^1(e_2,e_2)=-\left(\gamma^2+\alpha\beta\right),\\\notag
&\widetilde{\rho}^1(e_2,e_3)=-\frac{\alpha\gamma}{2},~~\widetilde{\rho}^1(e_3,e_3)=0.
\end{align}
By Lemma 2.18 in \cite{W3} and (2.12), we have for $(G_2,g,J,\nabla^1,V)$
\begin{align}
&(L^1_Vg)(e_1,e_1)=0,~~(L^1_Vg)(e_1,e_2)=\lambda_2\gamma,\\\notag
&(L^1_Vg)(e_1,e_3)=-\alpha\lambda_2+\gamma\lambda_3,~~(L^1_Vg)(e_2,e_2)=-2\gamma\lambda_1,\\\notag
&(L^1_Vg)(e_2,e_3)=\lambda_1\beta,~~(L^1_Vg)(e_3,e_3)=0.
\end{align}
If $(G_2,g,J,V)$ is an affine Ricci soliton associated to the connection $\nabla^1$, then by (2.14), we have
\begin{align}
\left\{\begin{array}{l}
-\beta^2-\gamma^2+\lambda=0,\\
\lambda_2\gamma=0,\\
-\alpha\lambda_2+\gamma\lambda_3=0,\\
-\gamma\lambda_1-\left(\gamma^2+\alpha\beta\right)+\lambda=0,\\
\lambda_1\beta-\alpha\gamma=0,\\
\lambda=0.\\
\end{array}\right.
\end{align}
Solve (2.28), we have
\vskip 0.5 true cm
\begin{thm}
$(G_2,g,J,V)$ is not an affine Ricci soliton associated to the connection $\nabla^1$.
\end{thm}
By (2.3) and Lemma 3.8 in \cite{BO}, we have for $G_3$, there exists a pseudo-orthonormal basis $\{e_1,e_2,e_3\}$ with $e_3$ timelike such that the Lie
algebra of $G_3$ satisfies
\begin{equation}
[e_1,e_2]=-\gamma e_3,~~[e_1,e_3]=-\beta e_2,~~[e_2,e_3]=\alpha e_1.
\end{equation}
By (2.64) in \cite{W3}, we have for $(G_3,g,J,\nabla^0)$
\begin{align}
&\widetilde{\rho}^0(e_1,e_1)=-\lambda a_3,~~\widetilde{\rho}^0(e_1,e_2)=0,\\\notag
&\widetilde{\rho}^0(e_1,e_3)=0,~~\widetilde{\rho}^0(e_2,e_2)=-\lambda a_3,\\\notag
&\widetilde{\rho}^0(e_2,e_3)=0,~~\widetilde{\rho}^0(e_3,e_3)=0,
\end{align}
where $a_3=\frac{1}{2}(\alpha+\beta-\gamma).$
By Lemma 2.24 in \cite{W3} and (2.12), we have for $(G_3,g,J,\nabla^0,V)$
\begin{align}
&(L^0_Vg)(e_1,e_1)=0,~~(L^0_Vg)(e_1,e_2)=0,\\\notag
&(L^0_Vg)(e_1,e_3)=-a_3\lambda_2,~~(L^0_Vg)(e_2,e_2)=0,\\\notag
&(L^0_Vg)(e_2,e_3)=a_3\lambda_1,~~(L^0_Vg)(e_3,e_3)=0.
\end{align}
If $(G_3,g,J,V)$ is an affine Ricci soliton associated to the connection $\nabla^0$, then by (2.13), we have
\begin{align}
\left\{\begin{array}{l}
\gamma a_3=0,\\
\lambda_2a_3=0,\\
\lambda_1 a_3=0,\\
\lambda=0.\\
\end{array}\right.
\end{align}
Solve (2.32), we have
\vskip 0.5 true cm
\begin{thm}
$(G_3,g,J,V)$ is an affine Ricci soliton associated to the connection $\nabla^0$ if and only if\\
(i) $\lambda=0$, $\alpha+\beta-\gamma=0$,\\
(ii)$\lambda=0$, $\alpha+\beta-\gamma\neq 0$, $\gamma=\lambda_1=\lambda_2=0$.
\end{thm}
By (2.69) in \cite{W3}, we have for $(G_3,g,J,\nabla^1)$
\begin{align}
&\widetilde{\rho}^1(e_1,e_1)=\lambda(a_1-a_3),~~\widetilde{\rho}^1(e_1,e_2)=0,\\\notag
&\widetilde{\rho}^1(e_1,e_3)=0,~~\widetilde{\rho}^1(e_2,e_2)=-\gamma(a_2+a_3),\\\notag
&\widetilde{\rho}^1(e_2,e_3)=0,~~\widetilde{\rho}^1(e_3,e_3)=0,
\end{align}
where $a_1=\frac{1}{2}(\alpha-\beta-\gamma).$
By Lemma 2.27 in \cite{W3} and (2.12), we have for $(G_3,g,J,\nabla^1,V)$
\begin{align}
&(L^1_Vg)(e_1,e_1)=0,~~(L^1_Vg)(e_1,e_2)=0,\\\notag
&(L^1_Vg)(e_1,e_3)=-(a_2+a_3)\lambda_2,~~(L^1_Vg)(e_2,e_2)=0,\\\notag
&(L^1_Vg)(e_2,e_3)=\lambda_1(a_3-a_1),~~(L^1_Vg)(e_3,e_3)=0.
\end{align}
If $(G_3,g,J,V)$ is an affine Ricci soliton associated to the connection $\nabla^1$, then by (2.14), we have
\begin{align}
\left\{\begin{array}{l}
\gamma(a_1-a_3)+\lambda=0,\\
(a_2+a_3)\lambda_2=0,\\
-\gamma(a_2+a_3)+\lambda=0,\\
\lambda_1(a_3-a_1)=0,\\
\lambda=0.\\
\end{array}\right.
\end{align}
Solve (2.35), we have
\vskip 0.5 true cm
\begin{thm}
$(G_3,g,J,V)$ is an affine Ricci soliton associated to the connection $\nabla^1$ if and only if\\
(i) $\lambda=0$, $\gamma\neq 0$, $\alpha=\beta=0$,\\
(ii)$\lambda=0$, $\gamma= 0$, $\alpha\lambda_2=0$, $\lambda_1\beta=0$.
\end{thm}
By (2.4) and Lemma 3.11 in \cite{BO}, we have for $G_4$, there exists a pseudo-orthonormal basis $\{e_1,e_2,e_3\}$ with $e_3$ timelike such that the Lie
algebra of $G_4$ satisfies
\begin{align}
[e_1,e_2]=-e_2+(2\eta-\beta)e_3,~~\eta=1~{\rm or}-1,~~[e_1,e_3]=-\beta e_2+ e_3,~~[e_2,e_3]=\alpha e_1.
\end{align}
By (2.81) in \cite{W3}, we have for $(G_4,g,J,\nabla^0)$
\begin{align}
&\widetilde{\rho}^0(e_1,e_1)=(2\eta-\beta)b_3-1,~~\widetilde{\rho}^0(e_1,e_2)=0,\\\notag
&\widetilde{\rho}^0(e_1,e_3)=0,~~\widetilde{\rho}^0(e_2,e_2)=(2\eta-\beta)b_3-1,\\\notag
&\widetilde{\rho}^0(e_2,e_3)=\frac{b_3-\beta}{2},~~\widetilde{\rho}^0(e_3,e_3)=0,
\end{align}
where $b_3=\frac{\alpha}{2}+\eta.$
By Lemma 2.32 in \cite{W3} and (2.12), we have for $(G_4,g,J,\nabla^0,V)$
\begin{align}
&(L^0_Vg)(e_1,e_1)=0,~~(L^0_Vg)(e_1,e_2)=-\lambda_2,\\\notag
&(L^0_Vg)(e_1,e_3)=-b_3\lambda_2,~~(L^0_Vg)(e_2,e_2)=2\lambda_1,\\\notag
&(L^0_Vg)(e_2,e_3)=b_3\lambda_1,~~(L^0_Vg)(e_3,e_3)=0.
\end{align}
If $(G_4,g,J,V)$ is an affine Ricci soliton associated to the connection $\nabla^0$, then by (2.13), we have
\begin{align}
\left\{\begin{array}{l}
(2\eta-\beta)b_3-1+\lambda=0,\\
\lambda_2=0,\\
\lambda_1+(2\eta-\beta)b_3-1+\lambda=0,\\
\lambda_1b_3+b_3-\beta=0,\\
\lambda=0.\\
\end{array}\right.
\end{align}
Solve (2.39), we have
\vskip 0.5 true cm
\begin{thm}
$(G_4,g,J,V)$ is an affine Ricci soliton associated to the connection $\nabla^0$ if and only if
$\lambda=\lambda_1=\lambda_2=0$, $\alpha=0$, $\beta=\eta.$
\end{thm}
By (2.89) in \cite{W3}, we have for $(G_4,g,J,\nabla^1)$
\begin{align}
&\widetilde{\rho}^1(e_1,e_1)=-[1+(\beta-2\eta)(b_3-b_1)],~~\widetilde{\rho}^1(e_1,e_2)=0,\\\notag
&\widetilde{\rho}^1(e_1,e_3)=0,~~\widetilde{\rho}^1(e_2,e_2)=-[1+(\beta-2\eta)(b_2+b_3)],\\\notag
&\widetilde{\rho}^1(e_2,e_3)=\frac{\alpha+b_3-b_1-\beta}{2},~~\widetilde{\rho}^1(e_3,e_3)=0,
\end{align}
where $b_1=\frac{\alpha}{2}+\eta-\beta$, $b_2=\frac{\alpha}{2}-\eta.$
By Lemma 2.36 in \cite{W3} and (2.12), we have for $(G_4,g,J,\nabla^1,V)$
\begin{align}
&(L^1_Vg)(e_1,e_1)=0,~~(L^1_Vg)(e_1,e_2)=-\lambda_2,\\\notag
&(L^1_Vg)(e_1,e_3)=-(b_2+b_3)\lambda_2-\lambda_3,~~(L^1_Vg)(e_2,e_2)=2\lambda_1,\\\notag
&(L^1_Vg)(e_2,e_3)=\lambda_1(b_3-b_1),~~(L^1_Vg)(e_3,e_3)=0.
\end{align}
If $(G_4,g,J,V)$ is an affine Ricci soliton associated to the connection $\nabla^1$, then by (2.14), we have
\begin{align}
\left\{\begin{array}{l}
-[1+(\beta-2\eta)(b_3-b_1)]+\lambda=0,\\
\lambda_2=0,\\
-(b_2+b_3)\lambda_2-\lambda_3=0,\\
\lambda_1-[1+(\beta-2\eta)(b_2+b_3)]+\lambda=0,\\
\lambda_1(b_3-b_1)+(\alpha+b_3-b_1-\beta)=0,\\
\lambda=0.\\
\end{array}\right.
\end{align}
Solve (2.42), we have
\vskip 0.5 true cm
\begin{thm}
$(G_4,g,J,V)$ is not an affine Ricci soliton associated to the connection $\nabla^1$.
\end{thm}
By (2.5) and Lemma 4.1 in \cite{BO}, we have for $G_5$, there exists a pseudo-orthonormal basis $\{e_1,e_2,e_3\}$ with $e_3$ timelike such that the Lie
algebra of $G_5$ satisfies
\begin{equation}
[e_1,e_2]=0,~~[e_1,e_3]=\alpha e_1+\beta e_2,~~[e_2,e_3]=\gamma e_1+\delta e_2,~~\alpha+\delta\neq 0,~~\alpha\gamma+\beta\delta=0.
\end{equation}
By (3.5) in \cite{W3}, we have for $(G_5,g,J,\nabla^0)$,
$\widetilde{\rho}^0(e_i,e_j)=0$, for $1\leq i,j\leq 3$.
By Lemma 3.3 in \cite{W3} and (2.12), we have for $(G_5,g,J,\nabla^0,V)$
\begin{align}
&(L^0_Vg)(e_1,e_1)=0,~~(L^0_Vg)(e_1,e_2)=0,\\\notag
&(L^0_Vg)(e_1,e_3)=\frac{\beta-\gamma}{2}\lambda_2,~~(L^0_Vg)(e_2,e_2)=0,\\\notag
&(L^0_Vg)(e_2,e_3)=-\frac{\beta-\gamma}{2}\lambda_1,~~(L^0_Vg)(e_3,e_3)=0.
\end{align}
If $(G_5,g,J,V)$ is an affine Ricci soliton associated to the connection $\nabla^0$, then by (2.13), we have
\begin{align}
\left\{\begin{array}{l}
\lambda=0,\\
(\beta-\gamma)\lambda_2=0,\\
(\beta-\gamma)\lambda_1=0,\\
\end{array}\right.
\end{align}
Solve (2.45), we have
\vskip 0.5 true cm
\begin{thm}
$(G_5,g,J,V)$ is an affine Ricci soliton associated to the connection $\nabla^0$ if and only if\\
(i)$\lambda=\beta=\gamma=0$, $\alpha+\delta\neq 0$.\\
(ii)$\lambda=0$, $\beta\neq \gamma$, $\lambda_1=\lambda_2=0$, $\alpha+\delta\neq 0$, $\alpha\gamma+\beta\delta=0$.
\end{thm}
By Lemma 3.7 in \cite{W3}, we have for $(G_5,g,J,\nabla^1)$, $\widetilde{\rho}^1(e_i,e_j)=0$, for $1\leq i,j\leq 3$.
By Lemma 3.6 in \cite{W3} and (2.12), we have for $(G_5,g,J,\nabla^1,V)$
\begin{align}
&(L^1_Vg)(e_1,e_1)=0,~~(L^1_Vg)(e_1,e_2)=0,\\\notag
&(L^1_Vg)(e_1,e_3)=-\alpha\lambda_1-\gamma\lambda_2,~~(L^1_Vg)(e_2,e_2)=0,\\\notag
&(L^1_Vg)(e_2,e_3)=-\beta\lambda_1-\delta\lambda_2,~~(L^1_Vg)(e_3,e_3)=0.
\end{align}
If $(G_5,g,J,V)$ is an affine Ricci soliton associated to the connection $\nabla^1$, then by (2.14), we have
\begin{align}
\left\{\begin{array}{l}
\lambda=0,\\
\alpha\lambda_1+\gamma\lambda_2=0,\\
\beta\lambda_1+\delta\lambda_2=0.\\
\end{array}\right.
\end{align}
\noindent Solve (2.47), we have
\vskip 0.5 true cm
\begin{thm}
$(G_5,g,J,V)$ is an affine Ricci soliton associated to the connection $\nabla^1$ if and only if\\
(i)$\lambda=\lambda_1=\lambda_2=0,$\\
(ii)$\lambda=0$, $\lambda_1\neq 0$, $\lambda_2=0$, $\alpha=\beta=0$, $\delta\neq 0$,\\
(iii)$\lambda=0$, $\lambda_1= 0$, $\lambda_2\neq 0$, $\delta=\gamma=0$, $\alpha\neq 0$.\\
\end{thm}
\indent By (2.6) and Lemma 4.3 in \cite{BO}, we have for $G_6$, there exists a pseudo-orthonormal basis $\{e_1,e_2,e_3\}$ with $e_3$ timelike such that the Lie
algebra of $G_6$ satisfies
\begin{equation}
[e_1,e_2]=\alpha e_2+\beta e_3,~~[e_1,e_3]=\gamma e_2+\delta e_3,~~[e_2,e_3]=0,~~\alpha+\delta\neq 0£¬~~\alpha\gamma-\beta\delta=0.
\end{equation}
\noindent By (3.18) in \cite{W3}, we have for $(G_6,g,J,\nabla^0)$
\begin{align}
&\widetilde{\rho}^0(e_1,e_1)=\frac{1}{2}\beta(\beta-\gamma)-\alpha^2,~~\widetilde{\rho}^0(e_1,e_2)=0,\\\notag
&\widetilde{\rho}^0(e_1,e_3)=0,~~\widetilde{\rho}^0(e_2,e_2)=\frac{1}{2}\beta(\beta-\gamma)-\alpha^2,\\\notag
&\widetilde{\rho}^0(e_2,e_3)=\frac{1}{2}[-\gamma\alpha+\frac{1}{2}\delta(\beta-\gamma)],~~\widetilde{\rho}^0(e_3,e_3)=0.
\end{align}
By Lemma 3.11 in \cite{W3} and (2.12), we have for $(G_6,g,J,\nabla^0,V)$
\begin{align}
&(L^0_Vg)(e_1,e_1)=0,~~(L^0_Vg)(e_1,e_2)=\alpha\lambda_2,\\\notag
&(L^0_Vg)(e_1,e_3)=\frac{\gamma-\beta}{2}\lambda_2,~~(L^0_Vg)(e_2,e_2)=-2\alpha\lambda_1,\\\notag
&(L^0_Vg)(e_2,e_3)=\frac{\beta-\gamma}{2}\lambda_1,~~(L^0_Vg)(e_3,e_3)=0.
\end{align}
If $(G_6,g,J,V)$ is an affine Ricci soliton associated to the connection $\nabla^0$, then by (2.13), we have
\begin{align}
\left\{\begin{array}{l}
\frac{1}{2}\beta(\beta-\gamma)-\alpha^2+\lambda=0,\\
\alpha\lambda_2=0,\\
({\gamma-\beta})\lambda_2=0,\\
-\alpha\lambda_1+\frac{1}{2}\beta(\beta-\gamma)-\alpha^2+\lambda=0,\\
\frac{\beta-\gamma}{2}\lambda_1-\gamma\alpha+\frac{1}{2}\delta(\beta-\gamma)=0,\\
\lambda=0.\\
\end{array}\right.
\end{align}
Solve (2.51), we have
\vskip 0.5 true cm
\begin{thm}
$(G_6,g,J,V)$ is an affine Ricci soliton associated to the connection $\nabla^0$ if and only if\\
(i)$\lambda=\lambda_1=\lambda_2=\gamma=\delta=0$, $\alpha\neq 0$, $\alpha^2=\frac{1}{2}\beta^2$,\\
(ii)$\lambda=\lambda_1=\lambda_2=\alpha=\beta=\gamma=0$, $\delta\neq 0$,\\
(iii)$\lambda=\lambda_2=0$, $\lambda_1\neq 0$, $\alpha=\beta=\gamma=0$, $\delta\neq 0$,\\
(iv)$\lambda=\lambda_2=0$, $\lambda_1\neq 0$, $\alpha=\beta=0$, $\delta\neq 0$, $\gamma\neq 0$, $\lambda_1=-\delta$,\\
(v)$\lambda=\alpha=\beta=\gamma=0$, $\lambda_2\neq 0$, $\delta\neq 0$.
\end{thm}
By (3.23) in \cite{W3}, we have for $(G_6,g,J,\nabla^1)$
\begin{align}
&\widetilde{\rho}^1(e_1,e_1)=-(\alpha^2+\beta\gamma),~~\widetilde{\rho}^1(e_1,e_2)=0,\\\notag
&\widetilde{\rho}^1(e_1,e_3)=0,~~\widetilde{\rho}^1(e_2,e_2)=-\alpha^2,\\\notag
&\widetilde{\rho}^1(e_2,e_3)=0,~~\widetilde{\rho}^1(e_3,e_3)=0.
\end{align}
By Lemma 3.15 in \cite{W3} and (2.12), we have for $(G_6,g,J,\nabla^1,V)$
\begin{align}
&(L^1_Vg)(e_1,e_1)=0,~~(L^1_Vg)(e_1,e_2)=\lambda_2\alpha,\\\notag
&(L^1_Vg)(e_1,e_3)=-\delta\lambda_3,~~(L^1_Vg)(e_2,e_2)=-2\alpha\lambda_1,\\\notag
&(L^1_Vg)(e_2,e_3)=-\gamma\lambda_1,~~(L^1_Vg)(e_3,e_3)=0.
\end{align}
If $(G_6,g,J,V)$ is an affine Ricci soliton associated to the connection $\nabla^1$, then by (2.14), we have
\begin{align}
\left\{\begin{array}{l}
-(\alpha^2+\beta\gamma)+\lambda=0,\\
\lambda_2\alpha=0,\\
\delta\lambda_3=0,\\
-\alpha\lambda_1-\alpha^2+\lambda=0,\\
\gamma\lambda_1=0,\\
\lambda=0.\\
\end{array}\right.
\end{align}
Solve (2.54), we have
\vskip 0.5 true cm
\begin{thm}
$(G_6,g,J,V)$ is an affine Ricci soliton associated to the connection $\nabla^1$ if and only if\\
(i)$\lambda=\alpha=\beta=\lambda_1=\lambda_3=0$, $\delta\neq 0$,\\
(ii)$\lambda=\alpha=\beta=\gamma=\lambda_3=0$, $\delta\neq 0$, $\lambda_1\neq 0$.
\end{thm}
By (2.7) and Lemma 4.5 in \cite{BO}, we have for $G_7$, there exists a pseudo-orthonormal basis $\{e_1,e_2,e_3\}$ with $e_3$ timelike such that the Lie
algebra of $G_7$ satisfies
\begin{equation}
[e_1,e_2]=-\alpha e_1-\beta e_2-\beta e_3,~~[e_1,e_3]=\alpha e_1+\beta e_2+\beta e_3,~~[e_2,e_3]=\gamma e_1+\delta e_2+\delta e_3,,~~\alpha+\delta\neq 0,~~\alpha\gamma=0.
\end{equation}
\noindent By (3.34) in \cite{W3}, we have for $(G_7,g,J,\nabla^0)$
\begin{align}
&\widetilde{\rho}^0(e_1,e_1)=-(\alpha^2+\frac{\beta\gamma}{2}),~~\widetilde{\rho}^0(e_1,e_2)=0,\\\notag
&\widetilde{\rho}^0(e_1,e_3)=-\frac{1}{2}(\gamma\alpha+\frac{\delta\gamma}{2}),~~\widetilde{\rho}^0(e_2,e_2)=
-(\alpha^2+\frac{\beta\gamma}{2}),\\\notag
&\widetilde{\rho}^0(e_2,e_3)=\frac{1}{2}(\alpha^2+\frac{\beta\gamma}{2}),~~\widetilde{\rho}^0(e_3,e_3)=0.
\end{align}
\noindent By Lemma 3.20 in \cite{W3} and (2.12), we have for $(G_7,g,J,\nabla^0,V)$
\begin{align}
&(L^0_Vg)(e_1,e_1)=-2\alpha\lambda_2,~~(L^0_Vg)(e_1,e_2)=\alpha\lambda_1-\beta\lambda_2,\\\notag
&(L^0_Vg)(e_1,e_3)=(\beta-\frac{\gamma}{2})\lambda_2,~~(L^0_Vg)(e_2,e_2)=2\beta\lambda_1,\\\notag
&(L^0_Vg)(e_2,e_3)=(\frac{\gamma}{2}-\beta)\lambda_1,~~(L^0_Vg)(e_3,e_3)=0.
\end{align}
If $(G_7,g,J,V)$ is an affine Ricci soliton associated to the connection $\nabla^0$, then by (2.13), we have
\begin{align}
\left\{\begin{array}{l}
-\alpha\lambda_2-(\alpha^2+\frac{\beta\gamma}{2})+\lambda=0,\\
\alpha\lambda_1-\beta\lambda_2=0,\\
(\beta-\frac{\gamma}{2})\lambda_2-(\gamma\alpha+\frac{\delta\gamma}{2})=0,\\
\beta\lambda_1-(\alpha^2+\frac{\beta\gamma}{2})+\lambda=0,\\
(\frac{\gamma}{2}-\beta)\lambda_1+\alpha^2+\frac{\beta\gamma}{2}=0,\\
\lambda=0.\\
\end{array}\right.
\end{align}
Solve (2.58), we have
\vskip 0.5 true cm
\begin{thm}
$(G_7,g,J,V)$ is an affine Ricci soliton associated to the connection $\nabla^0$ if and only if\\
(i)$\lambda=\alpha=\beta=\gamma=0$, $\delta\neq 0$,\\
(ii)$\lambda=\alpha=\beta=0$, $\gamma\neq 0$, $\lambda_1=0$, $\lambda_2=-\delta$, $\delta\neq 0$,\\
(iii)$\lambda=\alpha=\gamma=\lambda_1=\lambda_2=0$, $\beta\neq 0$.
\end{thm}
\vskip 0.5 true cm
By (3.42) in \cite{W3}, we have for $(G_7,g,J,\nabla^1)$
\begin{align}
&\widetilde{\rho}^1(e_1,e_1)=-\alpha^2,~~\widetilde{\rho}^1(e_1,e_2)=\frac{1}{2}(\beta\delta-\alpha\beta),\\\notag
&\widetilde{\rho}^1(e_1,e_3)=\beta(\alpha+\delta),~~\widetilde{\rho}^1(e_2,e_2)=-(\alpha^2+\beta^2+\beta\gamma),\\\notag
&\widetilde{\rho}^1(e_2,e_3)=\frac{1}{2}(\beta\gamma+\alpha\delta+2\delta^2),~~\widetilde{\rho}^1(e_3,e_3)=0.
\end{align}
\noindent By Lemma 3.24 in \cite{W3} and (2.12), we have for $(G_7,g,J,\nabla^1,V)$
\begin{align}
&(L^1_Vg)(e_1,e_1)=-2\alpha\lambda_2,~~(L^1_Vg)(e_1,e_2)=\alpha\lambda_1-\beta\lambda_2,\\\notag
&(L^1_Vg)(e_1,e_3)=-\alpha\lambda_1-\gamma\lambda_2-\beta\lambda_3,~~(L^1_Vg)(e_2,e_2)=2\beta\lambda_1,\\\notag
&(L^1_Vg)(e_2,e_3)=-\beta\lambda_1-\delta\lambda_2-\delta\lambda_3,~~(L^1_Vg)(e_3,e_3)=0.
\end{align}
If $(G_7,g,J,V)$ is an affine Ricci soliton associated to the connection $\nabla^1$, then by (2.14), we have
\begin{align}
\left\{\begin{array}{l}
-\alpha\lambda_2-\alpha^2+\lambda=0,\\
\alpha\lambda_1-\beta\lambda_2+\beta\delta-\alpha\beta=0,\\
-\alpha\lambda_1-\gamma\lambda_2-\beta\lambda_3+2\beta(\alpha+\delta)=0,\\
\beta\lambda_1-(\alpha^2+\beta^2+\beta\gamma)+\lambda=0,\\
-\beta\lambda_1-\delta\lambda_2-\delta\lambda_3+\beta\gamma+\alpha\delta+2\delta^2=0,\\
\lambda=0.\\
\end{array}\right.
\end{align}
Solve (2.61), we have
\vskip 0.5 true cm
\begin{thm}
$(G_7,g,J,V)$ is an affine Ricci soliton associated to the connection $\nabla^1$ if and only if\\
(i)$\lambda=\alpha=\beta=\gamma=0$, $\lambda_2+\lambda_3-2\delta=0$, $\delta\neq 0$,\\
(ii)$\lambda=\alpha=\beta=0$, $\gamma\neq 0$, $\lambda_2=0$, $\lambda_3=2\delta$, $\delta\neq 0$,\\
(iii)$\lambda=\alpha=0$, $\delta\neq 0$, $\beta\neq 0$, $\lambda_1=\beta+\gamma$, $\lambda_2=\delta$, $\lambda_3=\frac{-\gamma\delta+2\beta\delta}{\beta}$, $\gamma=\frac{\beta(\beta^2+\delta^2)}{\delta^2}$.
\end{thm}

\section{Affine Ricci solitons associated to perturbed canonical connections and perturbed Kobayashi-Nomizu connections on three-dimensional Lorentzian Lie groups}
    \indent We note that in our classifications in Section 2 always $\lambda=0.$ In order to get the affine Ricci soliton with non zero $\lambda$,
    we introduce perturbed canonical connections and perturbed Kobayashi-Nomizu connections in the following. Let $e_3^*$ be the dual base of $e_3$.
    We define on $G_{i=1,\cdots,7}$
  \begin{equation}
\nabla^2_XY=\nabla^0_XY+\overline{\lambda}e_3^*(X)e_3^*(Y)e_3,
\end{equation}
 \begin{equation}
\nabla^3_XY=\nabla^1_XY+\overline{\lambda}e_3^*(X)e_3^*(Y)e_3,
\end{equation}
where $\overline{\lambda}$ is a non zero real number. Then
 \begin{equation}
\nabla^2_{e_3}e_3=\overline{\lambda}e_3,~~\nabla^2_{e_i}e_j=\nabla^0_{e_i}e_j;
\end{equation}
 \begin{equation}
\nabla^3_{e_3}e_3=\overline{\lambda}e_3,~~\nabla^3_{e_i}e_j=\nabla^1_{e_i}e_j.
\end{equation}
where $i$ or $j$ does not equal $3$. We let
\begin{equation}
(L^j_Vg)(Y,Z):=g(\nabla^j_YV,Z)+g(Y,\nabla^j_ZV),
\end{equation}
 for $j=2,3$ and vector fields $V,Y,Z$. Then we have for $G_{i=1,\cdots,7}$
\begin{equation}
(L^2_Vg)(e_3,e_3)=-2\overline{\lambda}\lambda_3,~~(L^2_Vg)(e_j,e_k)=(L^0_Vg)(e_j,e_k),
\end{equation}
\begin{equation}
(L^3_Vg)(e_3,e_3)=-2\overline{\lambda}\lambda_3,~~(L^3_Vg)(e_j,e_k)=(L^1_Vg)(e_j,e_k),
\end{equation}
where $j$ or $k$ does not equal $3$.
\begin{defn}
$(G_i,g,J)$ is called the affine Ricci soliton associated to the connection $\nabla^2$ if it satisfies
\begin{equation}
(L^2_Vg)(Y,Z)+2\widetilde{\rho}^2(Y,Z)+2\lambda g(Y,Z)=0.
\end{equation}
$(G_i,g,J)$ is called the affine Ricci soliton associated to the connection $\nabla^3$ if it satisfies
\begin{equation}
(L^3_Vg)(Y,Z)+2\widetilde{\rho}^3(Y,Z)+2\lambda g(Y,Z)=0.
\end{equation}
\end{defn}
For $(G_1,\nabla^2)$, similar to (2.16), we have
\begin{equation}
\widetilde{\rho}^2(e_2,e_3)=\frac{\alpha^2+\overline{\lambda}\alpha}{2},~~~\widetilde{\rho}^2(e_j,e_k)=\widetilde{\rho}^0(e_j,e_k),
\end{equation}
for the pair $(j,k)\neq (2,3)$.
If $(G_1,g,J,V)$ is an affine Ricci soliton associated to the connection $\nabla^2$, then by (3.8), we have
\begin{align}
\left\{\begin{array}{l}
2\lambda_2\alpha-2\alpha^2-\beta^2+2\lambda=0,\\
\lambda_1\alpha=0,\\
-\beta\lambda_2+\alpha\beta=0,\\
-2\alpha^2-\beta^2+2\lambda=0,\\
\frac{\beta}{2}\lambda_1+\alpha^2+\overline{\lambda}\alpha=0,\\
\overline{\lambda}\lambda_3+\lambda=0.\\
\end{array}\right.
\end{align}
Solve (3.11), we have
\vskip 0.5 true cm
\begin{thm}
$(G_1,g,J,V)$ is an affine Ricci soliton associated to the connection $\nabla^2$ if and only if
$\lambda_1=\lambda_2=0$, $\lambda_3=-\overline{\lambda}$, $\alpha=-\overline{\lambda}$, $\beta=0$, $\lambda=\overline{\lambda}^2$.
\end{thm}
For $(G_1,\nabla^3)$, similar to (2.19), we have
\begin{equation}
\widetilde{\rho}^3(e_2,e_3)=\frac{\alpha^2+\overline{\lambda}\alpha}{2},~~~\widetilde{\rho}^3(e_j,e_k)=\widetilde{\rho}^1(e_j,e_k),
\end{equation}
for the pair $(j,k)\neq (2,3)$.
If $(G_1,g,J,V)$ is an affine Ricci soliton associated to the connection $\nabla^3$, then by (3.9), we have
\begin{align}
\left\{\begin{array}{l}
\lambda_2\alpha-\alpha^2-\beta^2+\lambda=0,\\
-\lambda_1\alpha+2\alpha\beta=0,\\
\lambda_1\alpha-\beta\lambda_2-\alpha\beta=0,\\
-\alpha^2-\beta^2+\lambda=0,\\
\beta\lambda_1-\alpha\lambda_2-\alpha\lambda_3+\alpha^2+\overline{\lambda}\alpha=0,\\
\overline{\lambda}\lambda_3+\lambda=0.\\
\end{array}\right.
\end{align}
Solve (3.13), we have
\vskip 0.5 true cm
\begin{thm}
$(G_1,g,J,V)$ is not an affine Ricci soliton associated to the connection $\nabla^3$.
\end{thm}
\begin{proof}By the first and second and fourth equations in (3.13) and $\alpha\neq 0$, we get $\lambda_2=0$, $\lambda_1=2\beta$, $\lambda=\alpha^2+\beta^2$, By the third equation in (3.13), we get $\lambda_1=\lambda_2=\beta=0$, $\lambda=\alpha^2$. By the fifth equation in (3.13), we get $\lambda_3=\alpha+\overline{\lambda}.$ By the sixth equation in (3,13), we get $\alpha^2+\overline{\lambda}\alpha+\overline{\lambda}^2=0.$ Then $\overline{\lambda}=\alpha=0$, this is a contradiction.
\end{proof}
For $(G_2,\nabla^2)$, similar to (2.23), we have
\begin{equation}
\widetilde{\rho}^2(e_1,e_3)=\frac{-\gamma\overline{\lambda}}{2},~~~\widetilde{\rho}^2(e_j,e_k)=\widetilde{\rho}^0(e_j,e_k),
\end{equation}
for the pair $(j,k)\neq (1,3)$.
If $(G_2,g,J,V)$ is an affine Ricci soliton associated to the connection $\nabla^2$, then by (3.8), we have
\begin{align}
\left\{\begin{array}{l}
-\left(\gamma^2+\frac{\alpha\beta}{2}\right)+\lambda=0,\\
\lambda_2\gamma=0,\\
\alpha\lambda_2+2\gamma\overline{\lambda}=0,\\
-\gamma\lambda_1-\left(\gamma^2+\frac{\alpha\beta}{2}\right)+\lambda=0,\\
\frac{\alpha}{2}\lambda_1+2(\frac{\beta\gamma}{2}-\frac{\alpha\gamma}{4})=0,\\
\overline{\lambda}\lambda_3+\lambda=0.\\
\end{array}\right.
\end{align}
Solve (3.15), we have
\vskip 0.5 true cm
\begin{thm}
$(G_2,g,J,V)$ is not an affine Ricci soliton associated to the connection $\nabla^2$.
\end{thm}
For $(G_2,\nabla^3)$, similar to (2.26), we have
\begin{equation}
\widetilde{\rho}^3(e_1,e_3)=\frac{-\gamma\overline{\lambda}}{2},~~~\widetilde{\rho}^3(e_j,e_k)=\widetilde{\rho}^1(e_j,e_k),
\end{equation}
for the pair $(j,k)\neq (1,3)$.
If $(G_2,g,J,V)$ is an affine Ricci soliton associated to the connection $\nabla^3$, then by (3.9), we have
\begin{align}
\left\{\begin{array}{l}
-\beta^2-\gamma^2+\lambda=0,\\
\lambda_2\gamma=0,\\
-\alpha\lambda_2+\gamma\lambda_3-\gamma\overline{\lambda}=0,\\
-\gamma\lambda_1-\left(\gamma^2+\alpha\beta\right)+\lambda=0,\\
\lambda_1\beta-\alpha\gamma=0,\\
\overline{\lambda}\lambda_3+\lambda=0.\\
\end{array}\right.
\end{align}
Solve (3.17), we have
\vskip 0.5 true cm
\begin{thm}
$(G_2,g,J,V)$ is not an affine Ricci soliton associated to the connection $\nabla^3$.
\end{thm}
For $(G_3,\nabla^2)$, we have
$\widetilde{\rho}^2(e_j,e_k)=\widetilde{\rho}^0(e_j,e_k),
$ for any pairs $(j,k)$.
If $(G_3,g,J,V)$ is an affine Ricci soliton associated to the connection $\nabla^2$, then by (3.8), we have
\begin{align}
\left\{\begin{array}{l}
-\gamma a_3+\lambda=0,\\
\lambda_2a_3=0,\\
\lambda_1 a_3=0,\\
\overline{\lambda}\lambda_3+\lambda=0.\\
\end{array}\right.
\end{align}
Solve (3.18), we have
\vskip 0.5 true cm
\begin{thm}
$(G_3,g,J,V)$ is an affine Ricci soliton associated to the connection $\nabla^2$ if and only if\\
(i) $a_3\neq 0$, $\lambda_1=\lambda_2=0$, $\lambda=\gamma a_3$, $\lambda_3=-\frac{\gamma a_3}{\overline{\lambda}}$,\\
(ii)$a_3=\lambda=\lambda_3=0$.
\end{thm}
For $(G_3,\nabla^3)$, we have
$\widetilde{\rho}^3(e_j,e_k)=\widetilde{\rho}^1(e_j,e_k),
$ for any pairs $(j,k)$.
If $(G_3,g,J,V)$ is an affine Ricci soliton associated to the connection $\nabla^3$, then by (3.9), we have
\begin{align}
\left\{\begin{array}{l}
\gamma(a_1-a_3)+\lambda=0,\\
(a_2+a_3)\lambda_2=0,\\
-\gamma(a_2+a_3)+\lambda=0,\\
\lambda_1(a_3-a_1)=0,\\
\overline{\lambda}\lambda_3+\lambda=0.\\
\end{array}\right.
\end{align}
Solve (3.19), we have
\vskip 0.5 true cm
\begin{thm}
$(G_3,g,J,V)$ is an affine Ricci soliton associated to the connection $\nabla^3$ if and only if\\
(i) $\gamma=\lambda=\lambda_3=0$, $\alpha\lambda_2=0$, $\beta\lambda_1=0$,\\
(ii)$\gamma\neq 0$, $\alpha=\beta=\lambda=\lambda_3=0$,\\
(iii)$\gamma\neq 0$, $\alpha=\beta\neq 0$, $\lambda_1=\lambda_2=0$, $\lambda=\alpha\gamma$, $\lambda_3=-\frac{\alpha\gamma}{\overline{\lambda}}$.
\end{thm}
For $(G_4,\nabla^2)$, we have
\begin{equation}
\widetilde{\rho}^2(e_1,e_3)=\frac{\overline{\lambda}}{2},~~~\widetilde{\rho}^2(e_j,e_k)=\widetilde{\rho}^0(e_j,e_k),
\end{equation}
for the pair $(j,k)\neq (1,3)$.
If $(G_4,g,J,V)$ is an affine Ricci soliton associated to the connection $\nabla^2$, then by (3.8), we have
\begin{align}
\left\{\begin{array}{l}
(2\eta-\beta)b_3-1+\lambda=0,\\
\lambda_2=0,\\
-b_3\lambda_2+\overline{\lambda}=0,\\
\lambda_1+(2\eta-\beta)b_3-1+\lambda=0,\\
\lambda_1b_3+b_3-\beta=0,\\
\overline{\lambda}\lambda_3+\lambda=0.\\
\end{array}\right.
\end{align}
Solve (3.21), we have
\vskip 0.5 true cm
\begin{thm}
$(G_4,g,J,V)$ is not an affine Ricci soliton associated to the connection $\nabla^2$.
\end{thm}
For $(G_4,\nabla^3)$, we have
\begin{equation}
\widetilde{\rho}^3(e_1,e_3)=\frac{\overline{\lambda}}{2},~~~\widetilde{\rho}^3(e_j,e_k)=\widetilde{\rho}^1(e_j,e_k),
\end{equation}
for the pair $(j,k)\neq (1,3)$.
If $(G_4,g,J,V)$ is an affine Ricci soliton associated to the connection $\nabla^3$, then by (3.9), we have
\begin{align}
\left\{\begin{array}{l}
-[1+(\beta-2\eta)(b_3-b_1)]+\lambda=0,\\
\lambda_2=0,\\
-(b_2+b_3)\lambda_2-\lambda_3+\overline{\lambda}=0,\\
\lambda_1-[1+(\beta-2\eta)(b_2+b_3)]+\lambda=0,\\
\lambda_1(b_3-b_1)+(\alpha+b_3-b_1-\beta)=0,\\
\overline{\lambda}\lambda_3+\lambda=0.\\
\end{array}\right.
\end{align}
Solve (3.23), we have
\vskip 0.5 true cm
\begin{thm}
$(G_4,g,J,V)$ is not an affine Ricci soliton associated to the connection $\nabla^3$.
\end{thm}
For $(G_5,g,J,\nabla^2)$,
$\widetilde{\rho}^2(e_i,e_j)=0$, for $1\leq i,j\leq 3$.
If $(G_5,g,J,V)$ is an affine Ricci soliton associated to the connection $\nabla^2$, then by (3.8), we have
\begin{align}
\left\{\begin{array}{l}
\lambda=0,\\
(\beta-\gamma)\lambda_2=0,\\
(\beta-\gamma)\lambda_1=0,\\
\overline{\lambda}\lambda_3+\lambda=0.\\
\end{array}\right.
\end{align}
Solve (3.24), we have
\vskip 0.5 true cm
\begin{thm}
$(G_5,g,J,V)$ is an affine Ricci soliton associated to the connection $\nabla^2$ if and only if\\
(i)$\gamma\neq \beta$, $\lambda=\lambda_1=\lambda_2=\lambda_3=0$, $\alpha+\delta\neq 0$, $\alpha\gamma+\beta\delta=0$,\\
(ii)$\lambda=\beta=\gamma=0$,  $\alpha+\delta\neq 0$, $\lambda_3=0$.
\end{thm}
For $(G_5,g,J,\nabla^3)$,
$\widetilde{\rho}^3(e_i,e_j)=0$, for $1\leq i,j\leq 3$.
If $(G_5,g,J,V)$ is an affine Ricci soliton associated to the connection $\nabla^3$, then by (3.9), we have
\begin{align}
\left\{\begin{array}{l}
\lambda=0,\\
\alpha\lambda_1+\gamma\lambda_2=0,\\
\beta\lambda_1+\delta\lambda_2=0,\\
\overline{\lambda}\lambda_3+\lambda=0.\\
\end{array}\right.
\end{align}
\noindent Solve (3.25), we have
\vskip 0.5 true cm
\begin{thm}
$(G_5,g,J,V)$ is an affine Ricci soliton associated to the connection $\nabla^3$ if and only if\\
(i)$\lambda=\lambda_1=\lambda_2=\lambda_3=0,$\\
(ii)$\lambda=\lambda_2=\lambda_3=\alpha=\beta=0$, $\lambda_1\neq 0$, $\delta\neq 0$,\\
(iii)$\lambda=0$, $\lambda_1=\lambda_3= 0$, $\lambda_2\neq 0$, $\delta=\gamma=0$, $\alpha\neq 0$.\\
\end{thm}
For $(G_6,\nabla^2)$, we have
\begin{equation}
\widetilde{\rho}^2(e_1,e_3)=\frac{\delta\overline{\lambda}}{2},~~~\widetilde{\rho}^2(e_j,e_k)=\widetilde{\rho}^0(e_j,e_k),
\end{equation}
for the pair $(j,k)\neq (1,3)$.
If $(G_6,g,J,V)$ is an affine Ricci soliton associated to the connection $\nabla^2$, then by (3.8), we have
\begin{align}
\left\{\begin{array}{l}
\frac{1}{2}\beta(\beta-\gamma)-\alpha^2+\lambda=0,\\
\alpha\lambda_2=0,\\
({\gamma-\beta})\lambda_2+2\delta\overline{\lambda}=0,\\
-\alpha\lambda_1+\frac{1}{2}\beta(\beta-\gamma)-\alpha^2+\lambda=0,\\
\frac{\beta-\gamma}{2}\lambda_1-\gamma\alpha+\frac{1}{2}\delta(\beta-\gamma)=0,\\
\overline{\lambda}\lambda_3+\lambda=0.\\
\end{array}\right.
\end{align}
Solve (3.27), we have
\vskip 0.5 true cm
\begin{thm}
$(G_6,g,J,V)$ is an affine Ricci soliton associated to the connection $\nabla^2$ if and only if\\
(i)$\alpha=\beta=0$, $\delta\neq 0$, $\gamma\neq 0$, $\lambda=\lambda_3=0$, $\lambda_1=-\delta$, $\lambda_2=-\frac{2\delta\overline{\lambda}}{\gamma}$,\\
(ii)$\alpha\neq 0$, $\lambda_1=\lambda_2=\gamma=\delta=0$, $\lambda=\alpha^2-\frac{1}{2}\beta^2$, $\lambda_3=-\frac{\lambda}{\overline{\lambda}}$.
\end{thm}
For $(G_6,\nabla^3)$, we have
\begin{equation}
\widetilde{\rho}^3(e_1,e_3)=\frac{\delta\overline{\lambda}}{2},~~~\widetilde{\rho}^3(e_j,e_k)=\widetilde{\rho}^1(e_j,e_k),
\end{equation}
for the pair $(j,k)\neq (1,3)$.
If $(G_6,g,J,V)$ is an affine Ricci soliton associated to the connection $\nabla^3$, then by (3.9), we have
\begin{align}
\left\{\begin{array}{l}
-(\alpha^2+\beta\gamma)+\lambda=0,\\
\lambda_2\alpha=0,\\
-\delta\lambda_3+\delta\overline{\lambda}=0,\\
-\alpha\lambda_1-\alpha^2+\lambda=0,\\
\gamma\lambda_1=0,\\
\overline{\lambda}\lambda_3+\lambda=0.\\
\end{array}\right.
\end{align}
Solve (3.29), we have
\vskip 0.5 true cm
\begin{thm}
$(G_6,g,J,V)$ is an affine Ricci soliton associated to the connection $\nabla^3$ if and only if $\alpha\neq 0$, $\lambda_1=\lambda_2=\gamma=\delta=0$, $\lambda=\alpha^2$, $\lambda_3=-\frac{\alpha^2}{\overline{\lambda}}$.
\end{thm}
For $(G_7,\nabla^2)$, we have
\begin{equation}
\widetilde{\rho}^2(e_1,e_3)=\frac{1}{2}(\beta\overline{\lambda}-\alpha\gamma-\frac{\delta\gamma}{2}),
\widetilde{\rho}^2(e_2,e_3)=\frac{1}{2}(\delta\overline{\lambda}+\alpha^2+\frac{\beta\gamma}{2}),
~~~\widetilde{\rho}^2(e_j,e_k)=\widetilde{\rho}^0(e_j,e_k),
\end{equation}
for the pair $(j,k)\neq (1,3),(2,3)$.
If $(G_7,g,J,V)$ is an affine Ricci soliton associated to the connection $\nabla^2$, then by (3.8), we have
\begin{align}
\left\{\begin{array}{l}
-\alpha\lambda_2-(\alpha^2+\frac{\beta\gamma}{2})+\lambda=0,\\
\alpha\lambda_1-\beta\lambda_2=0,\\
(\beta-\frac{\gamma}{2})\lambda_2+\beta\overline{\lambda}-(\gamma\alpha+\frac{\delta\gamma}{2})=0,\\
\beta\lambda_1-(\alpha^2+\frac{\beta\gamma}{2})+\lambda=0,\\
(\frac{\gamma}{2}-\beta)\lambda_1+\delta\overline{\lambda}+\alpha^2+\frac{\beta\gamma}{2}=0,\\
\overline{\lambda}\lambda_3+\lambda=0.\\
\end{array}\right.
\end{align}
Solve (3.31), we have
\vskip 0.5 true cm
\begin{thm}
$(G_7,g,J,V)$ is an affine Ricci soliton associated to the connection $\nabla^2$ if and only if\\
(i)$\alpha=\beta=0$, $\gamma\neq 0$, $\delta\neq 0$, $\lambda=0$, $\lambda_1=-\frac{2\delta\overline{\lambda}}{\gamma}$, $\lambda_2=-\delta$, $\lambda_3=0$,\\
(ii)$\alpha\neq 0$, $\lambda_1=\lambda_2=\beta=\gamma=0$, $\lambda=\alpha^2$, $\delta\neq 0$, $\overline{\lambda}=-\frac{\alpha^2}{\delta}$, $\lambda_3=\delta$.
\end{thm}
For $(G_7,\nabla^3)$, we have
\begin{equation}
\widetilde{\rho}^3(e_1,e_3)=\alpha\beta+\beta\delta+\frac{\beta\overline{\lambda}}{2},~~
\widetilde{\rho}^3(e_2,e_3)=\frac{1}{2}(\beta\gamma+\alpha\delta+2\delta^2+\delta\overline{\lambda}),
~~~\widetilde{\rho}^3(e_j,e_k)=\widetilde{\rho}^1(e_j,e_k),
\end{equation}
for the pair $(j,k)\neq (1,3),(2,3)$.
If $(G_7,g,J,V)$ is an affine Ricci soliton associated to the connection $\nabla^3$, then by (3.9), we have
\begin{align}
\left\{\begin{array}{l}
-\alpha\lambda_2-\alpha^2+\lambda=0,\\
\alpha\lambda_1-\beta\lambda_2+\beta\delta-\alpha\beta=0,\\
-\alpha\lambda_1-\gamma\lambda_2-\beta\lambda_3+2\beta(\alpha+\delta+\frac{\overline{\lambda}}{2})=0,\\
\beta\lambda_1-(\alpha^2+\beta^2+\beta\gamma)+\lambda=0,\\
-\beta\lambda_1-\delta\lambda_2-\delta\lambda_3+\beta\gamma+\alpha\delta+2\delta^2+\delta\overline{\lambda}=0,\\
\overline{\lambda}\lambda_3+\lambda=0.\\
\end{array}\right.
\end{align}
Solve (3.33), we have
\vskip 0.5 true cm
\begin{thm}
$(G_7,g,J,V)$ is an affine Ricci soliton associated to the connection $\nabla^3$ if and only if\\
(i)$\lambda=\alpha=\beta=\gamma=\lambda_3=0$, $\delta\neq 0$, $\lambda_2=2\delta+\overline{\lambda}$,\\
(ii)$\alpha=\beta=\lambda=\lambda_2=\lambda_3=0$, $\gamma\neq 0$, $\delta\neq 0$, $\overline{\lambda}=-2\delta$,\\
(iii)$\alpha=\lambda=\lambda_3=0$, $\beta\neq 0$, $\delta\neq 0$, $\lambda_1=\beta+\gamma$, $\lambda_2=\delta$, $\overline{\lambda}=\frac{\gamma\delta-2\beta\delta}{\beta}$, $\gamma=\frac{\beta^3+\beta\delta^2}{\delta^2}$,\\
(iv)$\alpha\neq 0$, $\beta=\gamma=\delta=\lambda_1=\lambda_2=0$, $\lambda=\alpha^2$, $\lambda_3=-\frac{\alpha^2}{\overline{\lambda}}$,\\
(v)$\alpha\neq 0$, $\beta=\gamma=\lambda_1=\lambda_2=0$, $\lambda=\alpha^2$, $\delta\neq 0$, $\lambda_3=\alpha+2\delta+\overline{\lambda}$,
$\overline{\lambda}^2+(\alpha+2\delta)\overline{\lambda}+\alpha^2=0$.
\end{thm}
\begin{proof}
We know that $\alpha\gamma=0$ and $\alpha+\delta\neq 0$.\\
case i)$\alpha=0$, then $\delta\neq 0$. By (3.33), we have $\lambda=\lambda_3=0$ and
\begin{align}
\left\{\begin{array}{l}
\beta(\lambda_2-\delta)=0,\\
-\gamma\lambda_2+2\beta\delta+\beta\overline{\lambda}=0,\\
\beta\lambda_1-(\beta^2+\beta\gamma)=0,\\
-\beta\lambda_1-\delta\lambda_2+\beta\gamma+2\delta^2+\delta\overline{\lambda}=0.\\
\end{array}\right.
\end{align}
Case i)-a)$\beta=0$, then by (3.34), we have $\gamma\lambda_2=0$ and $\lambda_2=2\delta+\overline{\lambda}$.\\
Casei)-a)-1)$\gamma=0$, we get (i).\\
casei)-a)-2)$\gamma\neq 0$, we get $\lambda_2=0$ and $\overline{\lambda}=-2\delta$. So we have (ii).\\
Case i)-b)$\beta\neq 0$, then by (3.34), we have $\lambda_1=\beta+\gamma$, $\lambda_2=\delta$,
$\overline{\lambda}=\frac{\gamma\delta-2\beta\delta}{\beta}$. By the fourth equation in (3.34), we get
 $\gamma=\frac{\beta^3+\beta\delta^2}{\delta^2}$ and this is (iii).\\
 Case ii) $\alpha\neq 0$, so $\gamma=0$.\\
 Case ii)-a) $\beta=0$, by (3.33), we get $\lambda_1=\lambda_2=0$, $\lambda=\alpha^2$, $\delta\lambda_3=\alpha\delta+2\delta^2+\delta\overline{\lambda}$, $\lambda_3=-\frac{\alpha^2}{\overline{\lambda}}$.\\
 Case ii)-a)-1) $\delta=0$, we get (iv).\\
 Case ii)-a)-2) $\delta\neq 0$, we get (v).\\
 Case ii)-b) $\beta\neq 0$, we get $\alpha\lambda_2+\beta\lambda_1-\beta^2=0$ and $\alpha\lambda_1-\beta\lambda_2+\beta(\delta-\alpha)=0$. So get
 \begin{align}
 \left\{\begin{array}{l}
\lambda_1=\beta-\frac{\alpha\beta\delta}{\alpha^2+\beta^2},\\
\lambda_2=\frac{\beta^2\delta}{\alpha^2+\beta^2},\\
\lambda=\frac{\alpha\beta^2\delta}{\alpha^2+\beta^2}+\alpha^2,\\
\lambda_3=\overline{\lambda}+\alpha+2\delta+\frac{\alpha^2\delta}{\alpha^2+\beta^2}.\\
\end{array}\right.
\end{align}
 Using (3.35) and the fifth equation in (3.33), we get $\beta^2(\alpha^2-\alpha\delta+\delta^2+\beta^2)+\alpha^2\delta^2=0$, so we get $\beta=0$ and this is a contradiction. So we have no solutions in this case.
 \end{proof}
\vskip 1 true cm

\section{Acknowledgements}

The author was supported in part by NSFC No.11771070.

\vskip 1 true cm

%-----------------------------------------------------------------------------
%-----------------------------------------------------------------------------

\bigskip
\bigskip

\noindent {\footnotesize {\it Y. Wang} \\
{School of Mathematics and Statistics, Northeast Normal University, Changchun 130024, China}\\
{Email: wangy581@nenu.edu.cn}


\begin{thebibliography}{20}

\bibitem{BO}
W. Batat, K. Onda, {\it Algebraic Ricci solitons of three-dimensional Lorentzian Lie groups}, J. Geom. Phys. 114 (2017), 138-152.
\bibitem{Ca0}
G. Calvaruso, {\it Three-dimensional homogeneous generalized Ricci solitons,} Mediterr. J. Math. 14 (2017), no. 5, Paper No. 216, 21 pp.
\bibitem{Ca1} G. Calvaruso, {\it Homogeneous structures on three-dimensional Lorentzian manifolds}, J. Geom. Phys 57 (4) (2007) 1279-1291.
\bibitem{CP} L.A. Cordero, P.E. Parker, {\it Left-invariant Lorentzian metrics on 3-dimensional Lie groups}, Rend. Mat. Appl. (7) 17 (1997) 129-155.
\bibitem{Cr}M. Crasmareanu, {\it A new approach to gradient Ricci solitons and generalizations}, Filomat 32 (2018), no. 9, 3337-3346.
\bibitem{ES}
F. Etayo, R. Santamaria, {\it Distinguished connections on $(J^2=\pm 1)$-metric manifolds}, Arch. Math. (Brno) 52 (2016), no. 3, 159-203.
\bibitem{HD}N. Halammanavar, K. Devasandra,
{\it Kenmotsu manifolds admitting Schouten-van Kampen connection},
Facta Univ. Ser. Math. Inform. 34 (2019), no. 1, 23-34.
\bibitem{Ha}R. Hamilton, {\it The Ricci flow on surfaces}, Contemp. Math., 71, Amer. Math. Soc., Providence, RI, 1988, 237-262.
\bibitem{HDZ}Y. Han, A. De, P. Zhao, {\it On a semi-quasi-Einstein manifold}, J. Geom. Phys. 155 (2020), 103739, 8 pp.
\bibitem{HPC}S. Hui, R. Prasad, D. Chakraborty, {\it Ricci solitons on Kenmotsu manifolds with respect to quarter symmetric non-metric $\phi$-connection}, Ganita 67 (2017), no. 2, 195-204.
\bibitem{PY}S. Perktas, A. Yildiz, {\it On quasi-Sasakian 3-manifolds with respect to the Schouten-van Kampen connection},
Int. Elec, J. Geom. 13, 2020, 62-74.
\bibitem{QW}Q. Qu, Y. Wang, {\it Multiply warped products with a quarter-symmetric connection}, J. Math. Anal. Appl. 431 (2015), no. 2, 955-987.
\bibitem{SCB}A. Siddiqui, B. Chen, O. Bahadir, {\it Statistical solitons and inequalities for statistical warped product submanifolds},
Mathematics, 2019, 7, 797.
\bibitem{SO1}S. Sular, C. Ozgur, {\it Warped products with a semi-symmetric metric connection}, Taiwanese J. Math. 15 (2011), no. 4, 1701-1719.
\bibitem{SO2}S. Sular, C.Ozgur, {\it Warped products with a semi-symmetric non-metric connection}, Arab. J. Sci. Eng. 36 (2011), no. 3, 461-473.
\bibitem{W1}Y. Wang, {\it Multiply warped products with a semisymmetric metric connection}, Abstr. Appl. Anal. 2014, Art. ID 742371, 12 pp.
\bibitem{W2}Y. Wang, {\it Curvature of multiply warped products with an affine connection}, Bull. Korean Math. Soc. 50 (2013), no. 5, 1567-1586.
\bibitem{W3}Y. Wang, {\it Canonical connections and algebraic Ricci solitons of three-dimensional Lorentzian Lie groups}, arxiv:2001.11656.

\end{thebibliography}
\end{document}